\documentclass[reqno, 12pt]{amsart}
\usepackage{amsmath}
 \usepackage{amssymb}
\usepackage{amsthm}
\usepackage{times}
\usepackage{latexsym}
\usepackage[mathscr]{eucal}

\numberwithin{equation}{section}
 
  \newtheorem{theorem}{Theorem}[section]

  \newtheorem{corollary}[theorem]{Corollary}

  \newtheorem{remark}[theorem]{Remark}

\title[$2$-Ruled hypersurfaces in a Walker $4$-manifold]{$2$-Ruled hypersurfaces in a Walker $4$-manifold}

\author[M. A. Dram\'e, A. Ndiaye, A. S. Diallo]{Mohamed Ayatola Dram\'e*, Ameth Ndiaye**, Abdoul Salam Diallo***}

\newcommand{\acr}{\newline\indent}

\address{\llap{*\,} Universit\'e Alioune Diop de Bambey\acr
UFR SATIC, D\'epartement de Math\'ematiques\acr
\'Equipe de Recherche en Analyse Non Lin\'eaire et G\'eom\'etrie (ER ANLG)\acr
B. P. 30, Bambey, S\'en\'egal\acr} 
\email{rabanyayatoulah@gmail.com}
\address{\llap{**\,} Universit\'e Cheik Anta Diop de Dakar\acr
FASTEF, D\'epartement de Math\'ematiques\acr
B. P. 880, Dakar, S\'en\'egal \acr}
\email{ameth1.ndiaye@ucad.edu.sn}

\address{\llap{***\,}  Universit\'e Alioune Diop de Bambey\acr
UFR SATIC, D\'epartement de Math\'ematiques\acr
\'Equipe de Recherche en Analyse Non Lin\'eaire et G\'eom\'etrie (ER ANLG)\acr
B. P. 30, Bambey, S\'en\'egal\acr} 
\email{abdoulsalam.diallo@uadb.edu.sn}
\thanks{}



\begin{document}
 \tiny{
\begin{abstract}  
The hypersurface is one of the most important objects in a space. Many 
authors studied diffrent geometric aspects of hypersurfaces in a space.
In this paper, we define three types of $2$-ruled hypersurfaces in 
a Walker $4$-manfold $\mathbb{E}^4_1$. We obtain the Gaussian and mean curvatures of the $2$-ruled hypersurfaces of type-$1$,  type-$2$ and 
type $3$. We give some  characterizations about its minimality. We also deal with the first Laplace-Beltrami operators of these types of $2$-ruled hypersurfaces in the considered Walker $4$-manifold.
\end{abstract}

\maketitle

\section{Introduction}\label{Introduction} 

\noindent
The study of hypersurface of a given ambiant space is a naturel interesting problem which enriches our knowledge and understanding of the geometry 
of the space itself. The theory of ruled surfaces in $\mathbb{R}^3$ is a 
classical subject in diffrential geometry. The study of ruled surfaces of 
a given ambiant space is a naturel and interesting problem. 
A surface $\Sigma$ in $M$ is said to be ruled if every point of $\Sigma$ is on 
(a open geodesic segment) in $M$ that lies in $\Sigma$ (see \cite{Nomizu94}). Locally a ruled surface is made by a one parameter family of geodesic
segments \cite{Carmo76}.  Ruled surfaces are one-parameter set of lines and they are one of the important topics of classifical differential geometry. 
A ruled surface is defined as
\begin{eqnarray*}
\varphi(s,t)=\alpha(s) + tX(s), s,t\in I\subset\mathbb{R},
\end{eqnarray*}
where the curve $\alpha(s)$ is called base curve and $X(s)$ is called the ruling 
of the ruled surface. A lots of studies have been done about different characterizations of ruled surfaces in $3$-dimensional Euclidean, Minkowskian, Galilean and pseudo-Galilean space (see \cite{Dillen, Divjak, Flory, Guler, Kim2004, Niang03} and references therein). In \cite{Aslan21-1}, the authors define a quaternionic operator whose scalar part is a real parameter and vector part is a curve in three dimensional real vector space $\mathbb{R}^3$. They prove that quaternion product of this operator and a spherical curve represents a ruled surface in $\mathbb{R}^3$ if the vector part of the quaternionic operator is perpendicular to the position vector of the spherical curve. Also in \cite{Aslan21-2}, the authors show that the split quaternion product of a split quaternion operator and a curve, which lies on Lorentzian unit sphere or on hyperbolic unit sphere, parametrizes a ruled surface in the $3$-dimensional Minkowski space $\mathbb{E}^{3}_{1}$ if the vector part of the operator is perpendicular to the position vector of the spherical curve. Recently, in \cite{Niang21}, the authors construct two special families of ruled surfaces in a three dimensional strict Walker manifold. They show that the local degeneracy (resp. non-degeneracy) to one of this family has a strong consequence on the geometry of the ambiant Walker manifold. Ruled hypersurfaces in higher dimensions have also been studied by many authors \cite{Barbosa84-1,Barbosa84-2}. In \cite{Boeckx96},  the intrinsic classification of irreducible ruled hypersurfaces of $\mathbb{R}^4$ has been given. In \cite{Kimura2020}, a new approach to investigating ruled real hypersurfaces in complex hyperbolic space $\mathbb{CH}^n$ is given.
In the paper \cite{Kimura2021}, the authors study ruled real hypersurfaces in the complex quadric.\\

\noindent
A $2$-ruled hypersurface in $\mathbb{R}^4$ is a one-parameter family of planes in $\mathbb{R}^4$. This is a generalization of ruled surfaces in 
$\mathbb{R}^3$. In \cite{Saji2002}, the author study singularities of 
$2$-ruled hypersurfaces in Euclidean $4$-space. After defining a 
non-degenerate $2$-ruled hypersurface, he gives a necessary and sufficient condition for such a map germ to be right-left equivalent to the cross cap 
$\times$ interval.  Also, the author in \cite{Saji2002} discusses the behavior 
of a generic $2$-ruled hypersurface map. In \cite{Altin2021} the authors 
obtain the Gauss map (unit normal vector field) of a $2$-ruled hypersurface 
in Euclidean $4$-space with the aid of its general parametric equation. They also obtain Gaussian and mean curvatures of the $2$-ruled hypersurface and they give some characterizations about its minimality. Finally, they deal with the first and second Laplace-Beltrami operators of $2$-ruled hypersurfaces in 
$\mathbb{E}^4$. Recently, in \cite{Ndiaye2022} the authors define three types 
of $2$-ruled hypersurfaces in the Minkowski $4$-space $E^{4}_{1}$. They 
obtain Gaussian and mean curvatures of the $2$-ruled hypersurfaces 
of type-$1$ and type-$2$, and some characterizations about its minimality. They also deal with the first Laplace-Beltrami operators of these types 
of $2$-ruled hypersurfaces in $E^{4}_{1}$.\\

\noindent
Motivated by the above two works, we study in this paper the $2$-ruled hypersurfaces in a Walker $4$-manifold. We define three types of $2$-ruled hypersurfaces  and we gives Gaussian and mean curvatures of the $2$-ruled hypersurface and some characterizations about its minimality. Our paper is organized as follows. Section \ref{Introduction} introduces the topic. 
In section \ref{Preliminaire},  we recall some basics notions on pseudo-Riemannian manifolds. In section \ref{Ruled}, we study $2$-ruled
ypersurface on a Walker $4$-manifold.

\section{Preliminaries}\label{Preliminaire}
\noindent
In this section, we recall some basics notions on pseudo-Riemannian manifolds taken from the book \cite{Falcitelli2004}. We begin with some algebraic preliminaries on non-degenerate bilinear forms on an $m$-dimensional real vector space $V$.\\

\noindent
Let $g : V \times V \to \mathbb{R}$ be a symmetric bilinear form. We say 
that $g$ is non-degenerate if $g(u, v) = 0$ for each $v \in V$ implies $u = 0$, otherwise $g$ is called degenerate. A non-degenerate symmetric bilinear 
form on $V$ is called a pseudo-Euclidean metric on $V$. It may induce either 
a non-degenerate or a degenerate symmetric bilinear form on a subspace 
$W$ of $V$; then $W$ is said to be a non-degenerate or a degenerate subspace, respectively. We say that $g$ is positive (negative) definite 
provided that $u \neq 0$ implies $g(u, u) > 0 (< 0)$. If $g$ is non-degenerate, there exists an ordered basis $(e_1, e_2,\ldots, e_m)$ of $V$ such that:
\begin{eqnarray*}
g(e_i,e_i) &=& -1, \quad 1 \leq i \leq q,\\
g(e_i,e_i) &=& 1, \quad q+1 \leq i \leq m,\\
g(e_i,e_j) &=& 0, \quad i\neq j,
\end{eqnarray*}
where $q$ is uniquely determined and $(q,m-q)$ is the signature of $g$. 
Obviously, in the case $q = 0$ or $s = m$, the first or the second condition 
has to be dropped. The integer $q$ is called the index of $g$ on $V$ and 
it is the largest dimension of a subspace $W \subset V$ on which the induced metric is negative definite.\\

\noindent
A pseudo-Riemannian metric $g$ on an $m$-dimensional manifold $M$ 
is a symmetric tensor field of type $(0,2)$ on $M$ such that for any 
$p \in M$ the tensor $g_p$ is a non-degenerate symmetric bilinear form 
on the tangent space  $T_p M$ of constant index. We call $(M,g)$ 
a pseudo-Riemannian manifold. Frequently, we denote by $M^{m}_{q}$ 
an $m$-dimensional pseudo-Riemannian manifold of index $q$. In the
particular case $m > 2$ and $q = 1$, we call $(M,g)$ a Lorentzian manifold.
Obviously, if $q = 0, (M, g)$ is a Riemannian manifold.\\

\noindent
Let $N^{n}_{s}$  be a submanifold of a pseudo-Riemannian manifold 
$M^{m}_{q}$. If the pseudo-Riemannian metric tensor $g_M$ of $M^{m}_{q}$
induces a pseudo-Riemannian metric tensor, a Riemannian metric tensor 
or a degenerate metric tensor $g_N$ on $N^{n}_{s}$ , then $N^{n}_{s}$ is called a pseudo-Riemannian  submanifold, a Riemannian submanifold or a degenerate submanifold, respectively, of $M^{m}_{q}$.
Let $M^{m}_{q}$ be an $m$-dimensional pseudo-Riemannian manifold with pseudo- Riemannian metric tensor $g_M$ of index $q$. Denoting by 
$\langle , \rangle$ the associated nondegenerate inner product on 
$M^{m}_{q}$, a tangent vector $X$ to $M^{m}_{q}$ is said to be spacelike
if $\langle X,X\rangle> 0$ ( or $X = 0$), timelike if $\langle X, X \rangle < 0$
or lightlike (null) if $\langle X, X\rangle = 0$ and $X \neq 0$. The set of null vectors of $T_p M$ is called the null cone at $p \in M$.\\

\noindent
Let $M^{m}_{1}(c)$ be an $m$-dimensional Lorentzian space form of constant curvature $c$, that is, $M^{m}_{1}(c)$ is the de Sitter space-time
$\mathbb{S}^{m}_{1}(c)$,  Minkowski space-time $\mathbb{R}^{4}_{1}(c)$ 
or the anti-de Sitter space-time $\mathbb{H}^{m}_{1}(c)$
according to $c > 0, c = 0$ or $c < 0$. For simplicity, we suppose that the constant curvature $c$ of  $M^{m}_{1}(c)$ is equal to $1, 0, -1$ according 
to whether $c > 0, c = 0, c < 0$.\\

\noindent
Now, we describe some basic examples of pseudo-Riemannian manifolds.
Let $\mathbb{R}^{m}_{q}$ be an $m$-dimensional pseudo-Euclidean space with metric tensor given by
\begin{eqnarray*}
g = -\sum_{i=1}^{q} (du_i)^2 + \sum_{i=q+1}^{m}(du_i)^2,
\end{eqnarray*}
where $(u_ 1,\ldots,u_m)$ is a coordinate system of $\mathbb{R}^{m}_{q}$.
So $(\mathbb{R}^{m}_{q}, g)$ is a ﬂat pseudo-Riemannian manifold of index 
$q$. Putting:
\begin{eqnarray*}
\mathbb{S}^{m}_{1}(1) =\{u\in \mathbb{R}^{m+1}_{1}, \langle u,u\rangle =1\},
\end{eqnarray*}
one obtains an $m$-dimensional pseudo-Riemannian manifold of index $q$ and of constant curvature $c = 1$. In the theory of general relativity,
 $\mathbb{S}^{4}_{1}(c)$ is called the de Sitter space-time. Putting:
 \begin{eqnarray*}
\mathbb{H}^{m}_{1}(-1) =\{u\in \mathbb{R}^{m+1}_{2}, \langle u,u\rangle =-1\},
\end{eqnarray*}
one obtains an $m$-dimensional pseudo-Riemannian manifold of index 
$q$ and of constant curvature $c=-1$. $\mathbb{H}^{m}_{1}(-1)$
is called  the anti-de Sitter space. 
We end this section by the following remark.

\begin{remark}
In contrast to the Riemannian case, there are topological obstructions to the existence of a Lorentz metric on a manifold $M$. Such a metric exists if 
either $M$ is non-compact, or $M$ is compact and has Euler number 
$\chi (M) = 0$.
\end{remark}

 \section{$2$-ruled hypersurfaces on a Walker $4$-manifold}\label{Ruled}
 
 \noindent
Hypersurfaces are one of the important objects in a space. Hypersurfaces in a manifold of constant curvature have been studied by many authors. Many
ambiant spaces are not always of constant curvature. In this paper we have studied $2$-ruled hypersurfaces in a Walker $4$-manifold.\\
 
 \noindent
A Walker $4$-manifold noted $M$, is a pseudo-Riemannian manifold, which admits a field of parallel null $2$-planes with signature $(++--)$. This class 
of manifold is locally isometric to $(U,g_f)$ where $U$ is an open of 
$\mathbb{R}^4$ and $g_f$ is the metric given, respectively to the local coordinates basis by 
$\{\partial_i=\frac{\partial}{\partial_{u_1}}\}_{i=1,2,3,4}$ by
\begin{eqnarray*}
g_f(\partial_1,\partial_3) &=& g_f(\partial_2,\partial_4)=1, \\
g_f(\partial_i,\partial_j) &=& g_{f_{ij}}(u_1,u_2,u_3,u_4) \quad \mbox{for} 
\quad i,j=3,4.
\end{eqnarray*}
The pseudo-Riemannian geometry of Walker metrics satisfying $g_{f_{34}}=0$ has been studied  by Chaichi et al. \cite{Chaichi05}. The purpose of this paper 
is to characterize some metrics  propertiers of Walker satisfying :
$g_{f_{33}}=g_{f_{44}}=0$. More precesily, we will consider Walker metrics 
of the following form:
\begin{eqnarray}\label{eq2.1}
g_f=
\left(
\begin{array}{cccc}
0 & 0 & 1 & 0 \cr
0 & 0 & 0 & 1 \cr
1 & 0 & 0 & f \cr
0 & 1 & f & 0 
\end{array}
\right),
\end{eqnarray}
where $f=f(u_3,u_4)$ denotes a differentiable function defined $U$. We denote by $f_3=\frac{\partial f(u_3,u_4)}{\partial u_3}$ and 
$f_4=\frac{\partial f(u_3,u_4)}{\partial u_4}$ for any function $f(u_3,u_4)$.
It follows after some straightforward calculations that the non zero 
christoffel symbols of a Walker metric (\ref{eq2.1}) are:
\begin{eqnarray*}
\Gamma_{33}^2 =f_3 \quad \mbox{and}\quad \Gamma_{44}^1=f_4.
\end{eqnarray*}
We deduce that the Levita-Civita connection of a Walker metric is given by 
\begin{eqnarray*}
    \nabla_{\partial_3}\partial_3 = f_3\partial_2 \quad \mbox{and}\quad \nabla_{\partial_4}\partial_4 = f_4\partial_1.
\end{eqnarray*}

\noindent
Let $\mathbb{R}^4 = \{(u_1, u_2, u_3, u_4 )|u_i \in \mathbb{R}, (i = 1, 2, 3,4)\}$ be an $4$-dimensional Cartesian space. For any
$u = (u_1, u_2, u_3, u_4 ), v = (v_1, v_2, v_3, v_4 ) \in \mathbb{R}^4$ , the pseudo-scalar product of $u$ and $v$ is defined by
\begin{eqnarray*}
\langle u, v \rangle = - u_1v_1 + \sum_{i=2}^{4}u_i v_i.
\end{eqnarray*}
We call $(\mathbb{R}^4 , \langle \cdot,\cdot \rangle)$ the Minkowski 
$4$-space. We shall write $\mathbb{R}^{4}_{1}$ instead of 
$(\mathbb{R}^4 , \langle \cdot,\cdot \rangle)$. We say that a non-zero vector 
$u \in \mathbb{R}^{4}_{1}$ is spacelike, lightlike or timelike if 
$\langle u,u \rangle > 0, \langle u,u \rangle = 0$ or $\langle u,u \rangle < 0$,
respectively. \\

\noindent
Since we will deal with $2$-ruled hypersurface in Walker $4$-manifold. We now define the Sitter $3$-space, the Hyperbolic $3$-space and the light cone 
at the origin, respectivily,  by
\begin{eqnarray}\label{eq2.2}
\mathbb{S}^{3}_{1} &=& \{ x\in M, \| u \| = 1\}, \label{eq2.2}\\
\mathbb{H}^{3}_{+}(-1) &=& \{ u\in M, \| u\| = -1\}, \label{eq2.2}\\
\mathcal{LC} &=& \{ x\in M, \| u \| = 0\},\label{eq2.4}
\end{eqnarray} 
where $\|u\|=\sqrt{g_f(u,u)}$.\\

\noindent
If $\overrightarrow{u}=(u_1,u_2,u_3,u_4) , 
\overrightarrow{v}=(v_1,v_2,v_3,v_4)$ and 
$\overrightarrow{w}=(w_1,w_2,w_3,w_4)$ are three vectors in $M$, then te vector product are defined by
\begin{eqnarray}\label{eq2.5}
\overrightarrow{u}\times_f \overrightarrow{v}\times_f
\overrightarrow{w}=
\left[
    \begin{array}{cccc}
      0 & -f & 1 & 0 \\
      -f & 0 & 0 & 1 \\
      1 & 0 & 0 & 0 \\
      0 & 1 & 0 & 0
     \end{array}
     \
\right]\det\left[
\begin{array}{cccc}
\partial_1 & \partial_2 & \partial_3 & \partial_4 \\
u_1 & u_2 & u_3 & u_4 \\
v_1 & v_2 & v_3 & v_4 \\
w_1 & w_2 & w_3 & w_4
\end{array}
\right].
  \end{eqnarray}
If
\begin{eqnarray*}
\varphi : I_1\times I_2\times I_3 &\to & M \\
(u_1,u_2,u_3) &\mapsto & \varphi(u_1,u_2,u_3),
\end{eqnarray*}
with
\begin{eqnarray}\label{eq2.6}
\varphi(u_1,u_2,u_3)=(\varphi_1(u_1,u_2,u_3),\varphi_2(u_1,u_2,u_3),\varphi_3(u_1,u_2,u_3),\varphi_4(u_1,u_2,u_3))
\end{eqnarray}
is a hypersurface in $M$ , then the Gauss map (i.e., the unit normal vector field), the matrix forms of the first and second fundamental forms are
\begin{eqnarray}\label{eq2.7}
G_f = \frac{\varphi_{u_1}\times_f\varphi_{u_2}\times_f\varphi_{u_2}}{\Vert\varphi_{u_1}\times_f\varphi_{u_2}\times\varphi_{u_3}\Vert},
\end{eqnarray}
\begin{eqnarray}\label{eq2.8}
[g_{ij}]=
\left[
   \begin{array}{ccc}
    g_{11} & g_{12} & g_{13} \\
    g_{21} & g_{22} & g_{23} \\
    g_{31} & g_{32} & g_{33}
   \end{array}
\right]
\end{eqnarray}
and
\begin{eqnarray}\label{eq2.9}
[h_{ij}]=
\left[
  \begin{array}{ccc}
   h_{11} & h_{12} & h_{13} \\
   h_{21} & h_{22} & h_{23} \\
   h_{31} & h_{32} & h_{33}
  \end{array}
\right],
\end{eqnarray}
respectively, where the coefficients 
$g_{ij} =g_f(\varphi_{{u_i}},\varphi_{{u_j}}), 
h_{ij}=g_f(\varphi_{u_iu_j},G_f)_{ i,j \in \{1,2,3\}}$, with 
$\varphi_{uv}=\displaystyle\sum_{k=1}^4\Big\{\frac{\partial^2\varphi_k}{\partial v \partial u }
+ \displaystyle\sum_{ij}\Gamma_{ij}^k \frac{\partial\varphi_i}{\partial u}\frac{\partial\varphi_j}{\partial v}\Big\}\partial_k$. Also, the matrix of shape operator of the hypersurface $\varphi$ (\ref{eq2.6}) is
\begin{eqnarray}\label{eq2.10}
S_f= [s_{ij}] = [g^{ij}]\cdot[h_{ij}],
\end{eqnarray}
where $[g^{ij}]$ is the inverse matrix of $[g_{ij}]$. With aid of 
(\ref{eq2.8})-(\ref{eq2.10}), the Gaussian curvature and mean curvature of a hypersurface in $M$ are given by
\begin{eqnarray}\label{eq2.11}
K_f = \frac{\det[h_{ij}]}{\det[g_{ij}]}
\end{eqnarray}
and
\begin{eqnarray}\label{eq2.12}
3H_f=trace(S_f),
\end{eqnarray}
respectively.

\subsection{$2$-ruled hypersurfaces of type-$1$ in $M$}
In this section, we give the deﬁnition of $2$-ruled hypersurfaces of type 
$1$ and state some results on Gaussian and mean curvatures. By a 
$2$-ruled hypersurface of type-$1$ in $M$, we mean  a map
$\varphi:I_1\times I_2\times I_3\to M$ of the form
\begin{eqnarray}\label{eq3.1}
\varphi(u_1,u_2,u_3)=\alpha(u_1)+u_2\beta(u_1)+u_3\gamma(u_3),
\end{eqnarray}
where $\alpha: I_1\to \mathbb{R}^{4}_{1},  \beta: I_2\to \mathbb{S}^{3}_{1}$ and $\gamma : I_3 \to \mathbb{S}^{3}_{1}$ are smooth maps, where 
$\mathbb{S}^{3}_{1}$ is the Sitter $3$-space of $M$ and $I_1, I_2, I_3$ are open intervals. We call $\alpha$ 
a base curve and two curves $\beta$ and $\gamma$ director curves. The planes 
$ (u_2,u_3)\to \alpha(u_1)+u_2\beta(u_1)+u_3\gamma(u_1)$ are called 
rulings \cite{Saji2002}.\\

\noindent
Putting:
\begin{eqnarray}\label{eq3.2}
\left \{
 \begin{array}{llllll}
\alpha(u_1) &=& \big(\alpha_{1}(u_1),\alpha_2(u_1),\alpha_3(u_1),\alpha_4(u_1)\big)  \\
 \beta(u_1) &=& \big(\beta_{1}(u_1),\beta_2 (u_1), \beta_3(u_1),
 \beta_4(u_1)\big)   \\
  \gamma(u_1) &=& \big(\gamma_1(u_1),\gamma_2(u_1),\gamma_3(u_1),\gamma_4(u_1)\big),
\end{array}
\right.
\end{eqnarray}
then, the equation (\ref{eq3.1}) becomes:
\begin{eqnarray}\label{eq3.3}
\varphi(u_1,u_2,u_3) =\left(
\begin{array}{cc}
\alpha_1(u_1)+u_2\beta_1(u_1)+u_3\gamma_1(u_1), & \alpha_2(u_1)+u_2\beta_2(u_1)+u_3\gamma_2(u_1),\\
\alpha_3(u_1)+u_2\beta_3(u_1)+u_3\gamma_3(u_1), & \alpha_4(u_1)+u_2\beta_4(u_1)+u_3\gamma_4(u_1)
\end{array}
  \right).
\end{eqnarray}
We see that $\langle\beta_{i},\beta_{i}\rangle=\langle\gamma_{i},\gamma_{i}\rangle=1$ and we state $\alpha_i=\alpha_i(u_1)$, $\beta_i=\beta_i(u_1)$, $\gamma_i=\gamma_i(u_1)$, $\varphi_i=\varphi_i(u_1,u_2,u_3)$, $f'=\frac{\partial f(u_1)}{\partial u_1}$, $f''=\frac{\partial^2f(u_1)}{\partial u_1\partial u_1}$, $i\in \{1,2,3,4\}$ and $f\in \{\alpha,\beta,\gamma\}$.
We denote by
\begin{eqnarray}
E_{ij} &=& \gamma_i(\alpha'_j+u_2\beta'_j+u_3\gamma'_j) \label{eq3.4}\\
F_{ij} &=& \beta_i(\alpha'_j+u_2\beta'_j+u_3\gamma'_j).\label{eq3.5}
\end{eqnarray}

\noindent
Now, let us prove the following theorem which contains the Gauss map of the $2$-ruled hypersurface of type-$1$ (\ref{eq3.3}).
\begin{theorem}\label{thmGauss1}
The Gauss map of the $2$-ruled hypersurface of type-$1$ of the form
(\ref{eq3.3}) is iven by
\begin{eqnarray}\label{eq3.6}
G_f(u_1,u_2,u_3) = \frac{G_1(u_1,u_2,u_3)\partial_1
+ G_2(u_1,u_2,u_3)\partial_2 + G_3(u_1,u_2,u_3)\partial_3 
+G_4(u_1,u_2,u_3)\partial_4}{A},
\end{eqnarray}
where
\begin{eqnarray}\label{eq3.7}
G_1(u_1,u_2,u_3) &=& -f\big(\beta_1(E_{43} - E_{34})
+ \beta_3(E_{14} - E_{41}) + \beta_4(E_{31} - E_{13})\big)\nonumber\\
&& + \beta_1(E_{24} - E_{42}) + \beta_2(E_{41} - E_{14})
+ \beta_4(E_{12} - E_{21}),\nonumber\\
G_2(u_1,u_2,u_3)&=& -f\big(\beta_2(E_{34} - E_{43}) 
+ \beta_3(E_{42} - E_{24}) + \beta_4(E_{23} - E_{32}\big)\nonumber\\
&& +\beta_1(E_{32} - E_{23}) + \beta_2(E_{13} - E_{31})
+ \beta_3(E_{21} - E_{12}),\nonumber\\
G_3(u_1,u_2,u_3)&=& \beta_2(E_{34} - E_{43}) + \beta_3(E_{42} - E_{24})
+ \beta_4(E_{23} - E_{32}),\nonumber\\
G_4(u_1,u_2,u_3)&=& \beta_1(E_{43} - E_{34}) + \beta_3(E_{14} - E_{41})
+\beta_4(E_{31} - E_{13}),
\end{eqnarray}
and
\begin{eqnarray}\label{eq3.8}
A=\sqrt{2G_1G_3 + 2G_2 G_4 + 2fG_3 G_4},
\end{eqnarray}
wit $G_1=G_1(u_1,u_2,u_3), G_2=G_2(u_1,u_2,u_3),G_3=G_3(u_1,u_2,u_3)$
and $G_4=G_4(u_1,u_2;u_3)$.
\end{theorem}

\begin{proof}
If we differentiate (\ref{eq3.3}), we get:
\begin{eqnarray*}
{\left\{ \begin{array}{ccc}
\varphi_{u_1}(u_1,u_2,u_3)&=&\big(\alpha'_1 + u_2\beta'_1 + u_3\gamma'_1,\alpha'_2 + u_2\beta'_2 + u_3\gamma'_2, \alpha'_3 + u_2\beta'_3 
+ u_3\gamma'_3, \alpha'_4 + u_2\beta'_4 + u_3\gamma'_4\big)\\
\varphi_{u_2}(u_1,u_2,u_3) &=& \big(\beta_1, \beta_2, \beta_3, \beta_4\big)\\
\varphi_{u_3}(u_1,u_2,u_3) &=& \big(\gamma_1, \gamma_2, \gamma_3, \gamma_4\big).
\end{array}\right.}
\end{eqnarray*}
By using the vector product define in (\ref{eq2.5}), we get:
\begin{eqnarray*}
\varphi_{u_1}\times_f \varphi_{u_2}\times_f \varphi_{u_3} &=&
\Big(-f\Big(\beta_1(E_{43} - E_{34}) + \beta_3(E_{14} -E_{41})
+ \beta_4(E_{31} - E_{13})\Big)\\
&& +\Big(\beta_1(E_{24} - E_{42}) + \beta_2(E_{41} - E_{14})
+\beta_4(E_{12} - E_{21})\Big)\Big)\partial_1\\
&& +\Big(-f \Big( \beta_2(E_{34} - E_{43}) + \beta_3(E_{42} - E_{24})
+\beta_4(E_{23} - E_{32}) \Big)\\
&& + \Big(\beta_1(E_{32} - E_{23}) + \beta_2(E_{13} - E_{31})
+\beta_3(E_{21} - E_{12})\Big)\Big)\partial_2\\
&&+\Big( \beta_2(E_{43}-E_{34})+\beta_3(E_{24}-E_{42})
+\beta_4(E_{32}-E_{23})\Big)\partial_3\\
&&+\Big( \beta_1(E_{43}-E_{34}) +\beta_3(E_{14}-E_{41})
+\beta_4(E_{31}-E_{13}) \Big)\partial_4.
\end{eqnarray*}
Now using the unit normal vector formula in (\ref{eq2.7}), we get the result.
\end{proof}

\noindent
From (\ref{eq2.8}), we obtain the matrix of the first fundamental form:
\begin{eqnarray}\label{eq3.9}
[g_{ij}]=
\left[
   \begin{array}{ccc}
    a & b & c \\ 
    b & 1 & e \\
    c & e & 1
   \end{array}
\right],
\end{eqnarray}
where
\begin{eqnarray}\label{eq3.10}
a &=& 2f(\alpha'_3 + u_2\beta'_3 + u_3\gamma'_3)(\alpha'_4 + u_2\beta'_4 + u_3\gamma'_4) \nonumber\\
&& +2\sum_{i=1}^2(\alpha'_i + u_2\beta'_i + u_3\gamma'_i)
(\alpha'_{i+2} + u_2\beta'_{i+2} + u_3\gamma'_{i+2}),\nonumber\\
b &=& f(F_{34} - F_{43}) + \sum_{i=1}^2(F_{i(i+2)} + F_{(i+2)i}),\nonumber\\
c &=& f(E_{34} - E_{43}) + \sum_{i=1}^2(E_{i(i+2)} + E_{(i+2)i)},\nonumber\\
e &=& f(\beta_3\gamma_4 + \beta_4\gamma_3)
+\sum_{i=1}^2(\beta_i\gamma_{i+2} + \beta_{i+2}\gamma_i),
\end{eqnarray}
and we obtain the inverse matrix $[g^{ij}]$ of $[g_{ij}]$ as:
\begin{eqnarray}\label{eq3.11}
[g^{ij}] = \frac{1}{\det[g_{ij}]}
\left[
   \begin{array}{ccc}
    1-e^2 & ce-b & be-c \\
  ce-b & a-c^2 &bc-ae \\
  be-c & bc-ae & a-b^2
   \end{array}
\right],
\end{eqnarray}
where
\begin{eqnarray}\label{eq3.12}
\det[g_{ij}] = -b^2 + 2cbe - c^2 - ae^2 + a = B.
\end{eqnarray}
Furthermore, from (\ref{eq2.9}), the matrix form of the second fundamental from of the $2$-ruled hypersurface (\ref{eq3.3}) is obtained by
\begin{eqnarray}\label{eq3.13}
[h_{ij}]= \left[ 
\begin{array}{ccc}
h_{11}  &  h_{12}  &  h_{13}  \cr
h_{21}  & h_{22}  &  h_{23}  \cr
h_{31}  &  h_{32}  &  h_{33}
\end{array}
\right],
\end{eqnarray}
where
\begin{eqnarray}\label{eq3.14}
\left \{ \begin{array}{llllll}
h_{11} &=& \displaystyle \frac{f_3G_4(\alpha_3'+u_2\beta_3'+u_3\gamma_3')+f_4G_3(\alpha_4'+u_2\beta_4'+u_3\gamma_ 4')+\sum^{2}_{i=1}G_{i+2}(\alpha_{i''} + u_2\beta_{i''} + u_3\gamma_{i''})}{\sqrt{2fG_3(u_1,u_2,u_3)G_4(u_1,u_2,u_3)+\sum^{2}_{i=1}G_i(u_1,u_2,u_3)G_{i+2}(u_1,u_2,u_3)}},\cr
h_{12}&= & h_{21}=\displaystyle\frac{f_3\beta_3G_4(\alpha'_3+u_2\beta'_3+u_3\gamma'_3)+f_4\beta_4G_3(\alpha'_4+u_2\beta'_4 - u_3\gamma'_4)+\sum_{i=1}^2 G_{i+2}\beta'_i}{\sqrt{2fG_3(u_1,u_2,u_3)G_4(u_1,u_2,u_3)+\sum^2_{i=1}G_i(u_1,u_2,u_3)G_{i+2}(u_1,u_2,u_3)}},\cr
h_{13}&=& h_{31}=\displaystyle\frac{f_3\gamma_3G_4(\alpha'_3+u_2\beta'_3+u_3\gamma'_3)+f_4\gamma_4G_3(\alpha'_4+u_2\beta'_4+u_3\gamma'_4)+\sum_{i=1}^2 G_{i+2}\gamma'_i}{\sqrt{2fG_3(u_1,u_2,u_3)G_4(u_1,u_2,u_3)+\sum^2_{i=1}G_i(u_1,u_2,u_3)G_{i+2}(u_1,u_2,u_3)}},\cr
h_{22}&=& \frac{f_3\beta^2_3G_4+f_4\beta^2_4G_3}{\sqrt{2fG_3(u_1,u_2,u_3)G_4(u_1,u_2,u_3)+\sum^2_{i=1}G_i(u_1,u_2,u_3)G_{i+2}(u_1,u_2,u_3)}},\cr
h_{33}&=& \frac{f_3\gamma^2_3G_4+f_4\gamma^2_4G_3}{\sqrt{2fG_3(u_1,u_2,u_3)G_4(u_1,u_2,u_3)+\sum^2_{i=1}G_i(u_1,u_2,u_3)G_{i+2}(u_1,u_2,u_3)}},\cr
h_{23}&=& \frac{f_3\beta_3\gamma_3G_4+f_4\beta_4\gamma_4G_3}{\sqrt{2fG_3(u_1,u_2,u_3)G_4(uç1,u_2,u_3)+\sum^2_{i=1}G_i(u_1,u_2,u_3)G_{i+2}(u_1,u_2,u_3)}}.
\end{array}
\right.
\end{eqnarray}
We can see easily that the $\det[h_{ij}]=h_{11}h_{22}h_{33}+2h_{12}h_{13}h_{23}-h^2_{12}h_{33}-h^2_{13}h_{22}-h^2_{23}h_{11}\ne0$. \\
Then we can give the following theorem by using (\ref{eq2.11})

\begin{theorem}
The $2$-ruled hypersurfaces of type-$1$ defined in (\ref{eq3.3}) is no flat.
\end{theorem}

\begin{corollary}
The $2$-ruled hypersurfaces of type-$1$ defined in (\ref{eq3.3}) is flat 
if $f$ is nonzero constant.
\end{corollary}

\begin{proof}
From (\ref{eq2.9}) the matrix of second fundamental form of the $2$-ruled hypersurface (\ref{eq3.3}) is obtained by
\begin{eqnarray}\label{eq3.15}
[h_{ij}]=
\left[
\begin{array}{ccc}
h_{11} & h_{12} & h_{13} \cr
h_{21} & 0 & 0 \cr
h_{31} & 0 & 0 
\end{array}
\right],
\end{eqnarray}
where
\begin{eqnarray}\label{eq3.16}
h_{11} &=& \displaystyle\frac{\sum^{2}_{i=1}G_{i+2}(\alpha_i'' + u_2\beta_i'' + u_3\gamma_i'')}{\sqrt{\sum^2_{i=1}G_i(u_1,u_2,u_3)G_{i+2}(u_1,u_2,u_3)}},\cr
h_{12} &=& h_{21}=\displaystyle\frac{\sum^2_{i=1}G_{i+2}\beta'_i}{\sqrt{\sum^2_{i=1}G_i(u_1,u_2,u_3)G_{i+2}(u_1,u_2,u_3)}},\cr
h_{13} &=& h_{31}\displaystyle\frac{\sum^2_{i=2}G_{i+2}\gamma'_i}{\sqrt{\sum^2_{i=1}G_i(u_1,u_2,u_3)G_{i+2}(u_1,u_2,u_3)}},\cr
h_{22} &=& h_{23}=h_{33}=0.
\end{eqnarray}
So we have $\det(h_{ij})=0$, hence $G_f=0$.
\end{proof}

\noindent
Now we will prove the following theorem about the mean curvature.
\begin{theorem}
The 2-ruled hypersurfaces of type-1 defined in (\ref{eq3.3}) is minimal, if
\begin{eqnarray}\label{eq3.17}
(1-e^2)\left[f_3G_4(\alpha_3'+u_2\beta_2'+u_3\gamma_3')+f_4G_3(\alpha_4'+u_2\beta_4'+u_3\gamma_4')+\sum^{2}_{i=1}G_{i+2}(\alpha_i''+u_2\beta_i''+u_3\gamma_i'')\right]\nonumber\\
+2(ce-b)\left[f_3\beta_3G_4(\alpha'_3+u_2\beta'_3+u_3\gamma'_3)+f_4\beta_4G_3(\alpha'_4+u_2\beta'_4+u_3\gamma'_4)+\sum_{i=1}^2 G_{i+2}\beta'_i\right]\nonumber\\
+2(be-c)\left[f_3\gamma_3G_4(\alpha'_3+u_2\beta'_3+u_3\gamma'_3)+f_4\gamma_4G_3(\alpha'_4+u_2\beta'_4+u_3\gamma'_4)+\sum_{i=1}^2 G_{i+2}\gamma'_i\right]\nonumber\\
+2(bc-ae)\left[f_3\beta_3\gamma_3G_4+f_4\beta_4\gamma_4G_3\right]
+(a-c^2)\left[f_3\beta^2_3G_4+f_4\beta^2_4G_3\right]\nonumber\\
+(a-b^2)\left[f_3\gamma^2_3G_4+f_4\gamma^2_4G_3\right]=0.
\end{eqnarray}
\end{theorem}

\begin{proof}
By (\ref{eq2.10}), the matrix of the shape operator is
\begin{eqnarray*}
S=
\left[
   \begin{array}{ccc}
    1-e^2 & ce-b & be-c \\
  ce-b & a-c^2 &bc-ae \\
  be-c & bc-ae & a-b^2
   \end{array}
\right]\left[
   \begin{array}{ccc}
    h_{11} & h_{12} & h_{13} \\
   h_{12} & h_{22} & h_ {23} \\
 h_{13} & h_{23} & h_{33}
   \end{array}
\right],
\end{eqnarray*} 
where $h_{11}$, $h_{12}$, $h_{13}$, $h_{22}$, $h_{23}$, $h_{33}$ are the same in (\ref{eq3.14}). Then we get the coefficients of $S$ by
\begin{eqnarray*}
S_{11}&=&(1-e^2)h_{11}+(ce-b)h_{12}+(be-c)h_{13}\\
S_{22}&=&(ce-b)h_{12}+(a-c^2)h_{22}+(bc-ae)h_{23}\\
S_{33}&=& (be-c)h_{13}+(bc-ae)h_{23}+(a-b^2)h_{33}.
\end{eqnarray*}
And using (\ref{eq3.14}) and (\ref{eq2.12}) we see that the $2$-ruled hypersurfaces is minimal if
\begin{eqnarray*}
S_{11}+S_{22}+S_{33}=0,
\end{eqnarray*}
then that end the proof.
\end{proof}

\begin{corollary}
If the curves $\beta$ and $\gamma$ are orthogonal then the $2$-ruled hypersurfaces of type-1 defined in (\ref{eq3.3}) is minimal if
\begin{eqnarray}\label{eq3.18}
\left[f_3G_4(\alpha_3'+u_2\beta_3'+u_3\gamma_3')+f_4G_3(\alpha'_4+u_2\beta'_4+u_3\gamma'_4)+\sum^{2}_{i=1}G_{i+2}(\alpha_i''+u_2\beta_i''+u_3\gamma_i'')\right]\nonumber\\
-2b\left[f_3\beta_3G_4(\alpha'_3+u_2\beta'_3+u_3\gamma'_4)+f_4  \beta_4G_3(\alpha'_4+u_2\beta'_4+\gamma'_4)+\sum_{i=1}^2 G_{i+2}\beta'_i\right]\nonumber\\
-2c\left[f_3\gamma_3G_4(\alpha'_3+u_2\beta'_3+u_3\gamma'_3)+f_4\gamma_4G_3(\alpha'_4+u_2\beta'_4+u_3\gamma'_4)+\sum_{i=1}^2 G_{i+2}\gamma'_i\right]\nonumber\\
+2bc\left[f_3\beta_3\gamma_3G_4 + f_4\beta_4\gamma_4G_3\right]
+(a-c^2)\left[f_3\beta^2_3G_4+f_4\gamma^2_4G_3\right]\nonumber\\
+(a-b^2)\left[f_3\gamma^2_3G_4+f_4\gamma^2_4G_3\right]=0.
\end{eqnarray}
\end{corollary}

\noindent
The Laplace-Beltrami operator of a smooth function 
$\varphi=\varphi(u_1,u_2,u_3)$ of class $C^3$ with respect to the first fundamental form of a hypersurface is defined as follows:
\begin{eqnarray}\label{eq3.19}
\Delta \varphi=\frac{1}{\sqrt{\det[g_{ij}]}}\sum_{i,j}^3\frac{\partial}{\partial u_i}\left(\sqrt{\det[g_{ij}]}g^{ij}\frac{\partial \varphi}{\partial u_j}\right).
\end{eqnarray}
Using (\ref{eq3.19}), we get the Laplace-Beltrami operator of the $2$-ruled hypersurface of type-$1$ (\ref{eq3.2}) by
\begin{eqnarray*}
\Delta\varphi=(\Delta\varphi_1,\Delta\varphi_2,\Delta\varphi_3,\Delta\varphi_4),
\end{eqnarray*}
where
\begin{eqnarray}\label{eq3.20}
\Delta\varphi_i =\frac{1}{\sqrt{ B} }\left[
\begin{aligned}
&   \frac{\partial}{\partial u_1}\left(\frac{(1-e^2)\varphi_{iu_1}+(ce-b)\varphi_{iu_2}+(be-c)\varphi_{iu_3}}{\sqrt{\det[g_{ij}]}}\right)\\
& + \frac{\partial}{\partial u_2}\left(\frac{(ce-b)\varphi_{iu_1}+(a-c^2)\varphi_{iu_2}+(bc-ae)\varphi_{iu_3}}{\sqrt{\det[g_{ij}]}}\right) \\
&+ \frac{\partial}{\partial u_3}\left(\frac{(be-c)\varphi_{iu_1}+(bc-ae)\varphi_{iu_2}+(a-b^2)\varphi_{iu_3}}{\sqrt{\det[g_{ij}]}}\right)
\end{aligned}
\right].
\end{eqnarray}
That is
\begin{eqnarray}\label{eq3.21}
\Delta\varphi_i=\frac{1}{\sqrt{ B} }\left[
\begin{aligned}
&   \frac{\partial}{\partial u_1}\left(\frac{(1-e^2)(\alpha'_i+u_2\beta'_i+u_3\gamma'_i)+(ce-b)\beta_i+(be-c)\gamma_i}{\sqrt{\det[g_{ij}]}}\right)\\
& + \frac{\partial}{\partial u_2}\left(\frac{(ce-b)(\alpha'_i+u_2\beta'_i+u_3\gamma'_i)+(a-c^2)\beta_i+(bc-ae)\gamma_i}{\sqrt{\det[g_{ij}]}}\right) \\
&+ \frac{\partial}{\partial u_3}\left(\frac{(be-c)(\alpha'_i+u_2\beta'_i+u_3\gamma'_i)+(bc-ae)\beta_i+(a-b^2)\gamma_i}{\sqrt{\det[g_{ij}]}}\right)
\end{aligned}
\right].
\end{eqnarray}
If we suppose that $\beta$ and $\gamma$ are orthogonal, then the 
Laplace-Beltrami operator of the $2$-ruled hypersuface of type-1 is given by
\begin{eqnarray}\label{eq3.22}
\Delta\varphi_i=\frac{1}{\sqrt{ a-b^2-c^2} }\left[
\begin{aligned}
&   \frac{\partial}{\partial u_1}\left(\frac{(\alpha'_i+u_2\beta'_i+u_3\gamma'_i)-b\beta_i-c\gamma_i}{\sqrt{a-b^2-c^2}}\right)\\
& + \frac{\partial}{\partial u_2}\left(\frac{-b(\alpha'_i+u_2\beta'_i+u_3\gamma'_i)+(a-c^2)\beta_i+bc\gamma_i}{\sqrt{a-b^2-c^2}}\right) \\
&+ \frac{\partial}{\partial u_3}\left(\frac{-c(\alpha'_i+u_2\beta'_i+u_3\gamma'_i)+bc\beta_i+(a-b^2)\gamma_i}{\sqrt{a-b^2-c^2}}\right)
\end{aligned}
\right].
\end{eqnarray}

\begin{theorem}
The components of the Laplace-Beltrami operator of the $2$-ruled 
hypersurface of type-$1$ are
\begin{eqnarray}\label{eq3.23}
\Delta\varphi_i=\frac{1}{\sqrt{W} }\left[
\begin{aligned}
&\frac{(\alpha''_i+u_2\beta''_i+u_3\gamma''_i-(b\beta_i)_{u_1}-(c\gamma_i)_{u_1})W-V_1(\alpha'_i+u_2\beta'_i+u_2\gamma'_i-b\beta_i-c\gamma_i )}{W^{\frac{3}{2}}}\\
& +\frac{(-b\beta'_i+((a-c^2)\beta_i)_{u_2}+(bc\gamma_i)_{u_2})W-V_2(-b(\alpha'_i+u_2\beta'_i+u_3\gamma'_i)+(a-c^2)\beta_i+bc\gamma_i)}{W^\frac{3}{2}} \\
&+ \frac{(-c\gamma'_i+(bc\beta_i)_{u_3}+((a-b^2)\gamma_i)_{u_3})W-V_3(-c(\alpha'_i+u_2\beta'_i+u_3\gamma'_i)+bc\beta_i+(a-b^2)\gamma_i)}{W^\frac{3}{2}}
\end{aligned}
\right],
\end{eqnarray}
where $i = 1, 2, 3, 4$; $\beta$ and $\gamma$ are orthogonal; $W=a-b^2-c^2$, $V_1=a_{u_1}-2bb_{u_1}-2cc_{u_1}$, $V_2=a_{u_2}-2bb_{u_2}-2cc_{u_2}$, $V_3=a_{u_3}-2bb_{u_3}-2cc_{u_3}$.
\end{theorem}

\subsection{2-Ruled hypersurfaces of type-2 in $M$}
A $2$-ruled hypersurface of type-2 in $M$  means (the image of) a map
$\varphi:I_1\times I_2\times I_3\to M$ of the form
\begin{eqnarray}\label{eq4.1}
\varphi(u_1,u_2,u_3)=\alpha(u_1)+u_2\beta(u_1)+u_3\gamma(u_1),
\end{eqnarray}
where $\alpha: I_1\to M$,  $\beta: I_2\to H^3_+(-1)$, 
$\gamma :I_3\to H^3_+(-1)$ are smooth maps, $H^3_+(-1)$ is the 
hyperbolic $3$-space of $M$ and $I_1, I_2, I_3$ are open intervals. We call 
$\alpha$ a base curve, $\beta$ and $\gamma$ director curves. The planes 
$ (u_2,u_3)\mapsto \alpha(u_1)+u_2\beta(u_1)+u_3\gamma(u_1)$ are called rulings. So, if we take
\begin{eqnarray}\label{eq4.2}
\left \{ \begin{array}{llllll}
\alpha(u_1) &=& \big(\alpha_1(u_1), \alpha_2(u_1), \alpha_3(u_1),
\alpha_4(u_1)\big)  \\
\beta(u_1) &=& \big(\beta_1 (u_1), \beta_2 (u_1), \beta_3(u_1),
\beta_4(u_1)\big)   \\
\gamma(u_1) &=& \big(\gamma_1(u_1), \gamma_2(u_1), \gamma_3(u_1),\gamma_4(u_1)\big)
\end{array}
\right.
\end{eqnarray}
in (\ref{eq4.1}), then we can write the $2$-ruled hypersurface of type-$2$ as
\begin{eqnarray}\label{eq4.3}
\varphi(u_1,u_2,u_3) = \left(
\begin{array}{cc}
\alpha_1(u_1)+u_2\beta_1(u_1)+u_3\gamma_1(u_1), & \alpha_2(u_1)+u_2\beta_2(u_1)+u_3\gamma_2(u_1),\\
\alpha_3(u_1)+u_2\beta_3(u_1)+u_2\gamma_3(u_1), & \alpha_4(u_1)+u_2\beta_4(u_1)+u_3\gamma_4(u_1)
\end{array}
  \right).
\end{eqnarray}
We see that $\langle\beta_i,\beta_i\rangle=\langle\gamma_i,\gamma_i\rangle=-1$ and we state $\alpha_i=\alpha_i(u_1)$, $\beta_i=\beta_i(u_1)$, $\gamma_i=\gamma_i(u_1)$, $\varphi_i=\varphi_i(u_1,u_2,u_3)$, $f'=\frac{\partial f(u_1)}{\partial u_1}$, $f''=\frac{\partial^2f(u_1)}{\partial u_1\partial u_1}$, $i\in \{1,2,3,4\}$ and 
$f\in \{\alpha,\beta,\gamma\}$.\\

\noindent
From (\ref{eq2.8}), we obtain the matrix of the first fundamental form
\begin{eqnarray}\label{eq4.4}
[g_{ij}]=
\left[
   \begin{array}{ccc}
    a & b & c \\
   b & -1 & e \\
   c & e & -1
   \end{array}
\right].
\end{eqnarray}
And we obtain the inverse matrix $[g^{ij}]$ of $[g_{ij}]$ as
\begin{eqnarray}\label{eq4.5}
[g^{ij}]=\frac{1}{\det[g_{ij}]}
\left[
   \begin{array}{ccc}
    1-e^2 & ce+b & be+c \\
  ce+b & -a-c^2 &bc-ae \\
  be+c & bc-ae & -a-b^2
   \end{array}
\right].
\end{eqnarray}
where
$a$, $b$, $c$ and $e$ are the same in (\ref{eq3.10}) and
\begin{eqnarray}\label{eq4.6}
\det[g_{ij}]=b^2+2cbe+c^2-ae^2+a=C.
\end{eqnarray}
Furthermore, from (\ref{eq2.9}), the matrix form of the second fundamental from of the $2$-ruled hypersurface (\ref{eq4.3}) is the same given in 
(\ref{eq3.13}) and (\ref{eq3.14}). And we have the following theorem since 
the $\det [h_{ij}]\ne0$.

\begin{theorem}
The $2$-ruled hypersurfaces of type-$2$ defined in (\ref{eq4.3}) is not flat.
\end{theorem}

\begin{corollary}
The $2$-ruled hypersurfaces of type-$2$ defined in (\ref{eq4.3}) is flat if 
$f$ is nonzero constant.
\end{corollary}

\noindent
For the mean curvature we have:

\begin{theorem}
The $2$-ruled hypersurfaces of type-$2$ defined in (\ref{eq4.3}) is minimal 
in $M$, if
\begin{eqnarray}\label{eq4.7}
(1-e^2)\left[f_3G_4(\alpha_3'+u_2\beta_3'+u_3\gamma_3')+f_4G_3(\alpha_4'+u_2\gamma_4'+u_3\gamma_4')+\sum^{2}_{i=1}G_{i+2}(\alpha_i''+u_2\beta_i''+u_3\gamma_i'')\right]\nonumber\\
+2(ce+b)\left[f_3\beta_3G_4(\alpha'_3+u_2\beta'_3+u_3\gamma'_3)+f_4\beta_4G_3(\alpha'_4+u_2\beta'_4+u_3\gamma'_4)+\sum_{i=1}^2 G_{i+2}\beta'_i+\right]\nonumber\\
+2(be+c)\left[f_3\gamma_3G_4(\alpha'_3+u_2\beta'_3+u_3\gamma'_3)+f_4\gamma_4G_3(\alpha'_4+u_2\beta'_4+u_3\gamma'_4)+\sum_{i=1}^2 G_{i+2}\gamma'_i\right]\nonumber\\
+2(bc-ae)\left[f_3\beta_3\gamma_3G_4 + f_4\beta_4\gamma_4G_3\right]
+(-a-c^2)\left[f_3\beta^2_3G_4 + f_4\beta^2_4G_3\right]\nonumber\\
+(-a-b^2)\left[f_3\gamma^2_3G_4 + f_4\gamma^2_4G_3\right]=0.
\end{eqnarray}
\end{theorem}

\begin{proof}
By (\ref{eq2.10}), the matrix of the shape operator is
\begin{eqnarray*}
S=
\left[
   \begin{array}{ccc}
    1-e^2 & ce+b & be+c \\
  ce+b & -a-c^2 &bc-ae \\
  be+c & bc-ae & -a-b^2
   \end{array}
\right]\left[
   \begin{array}{ccc}
    h_{11} & h_{12} & h_{13} \\
   h_{12} & h_{22} & h_{23} \\
 h_{13} & h_{23} & h_{33}
   \end{array}
\right],
\end{eqnarray*}
where $h_{11}$, $h_{12}$, $h_{13}$, $h_{22}$, $h_{23}$, $h_{33}$ are the same in (\ref{eq3.14}). Then we get the coefficients of $S$ by
\begin{eqnarray*}
S_{11}&=&(1-e^2)h_{11}+(ce+b)h_{12}+(be+c)h_{13},\\
S_{22}&=&(ce+b)h_{12}+(-a-c^2)h_{22}+(bc-ae)h_{23},\\
S_{33}&=&(be+c)h_{13}+(bc-ae)h_{23}+(-a-b^2)h_{33}.
\end{eqnarray*}
And using (\ref{eq3.14}) and (\ref{eq2.12}), we see that the $2$-ruled hypersurfaces of type-$2$ is minimal if
\begin{eqnarray*}
S_{11}+S_{22}+S_{33}=0,
\end{eqnarray*}
then that end the proof.
\end{proof}

\begin{corollary}
If the curves $\beta$ and $\gamma$ are orthogonal then the $2$-ruled hypersurfaces of type-$2$ defined in (\ref{eq4.3}) is minimal if
\begin{eqnarray}\label{eq4.8}
\left[f_3G_4(\alpha_3'+u_2\beta_3'+u_3\gamma_3')+f_4G_3(\alpha_4'+u_2\beta_4'+u_3\gamma_4')+\sum^{2}_{i=1}G_{i+2}(\alpha_i''+u_2\beta_i''+u_3\gamma_i'')\right]\nonumber\\
+2b\left[f_3\beta_3G_4(\alpha'_3+u_2\beta'_3+u_3\gamma'_3)+f_4\beta_4G_3(\alpha_4'+u_2\beta_4'+u_3\gamma_4')+\sum_{i=1}^2 G_{i+2}\beta'_i\right]\nonumber\\
+2c\left[f_3\gamma_3G_4(\alpha'_3+u_2\beta'_3+u_3\gamma'_3)+f_4\gamma_4G_3(\alpha'_4+u_2\beta'_2+u_3\gamma'_4)+\sum_{i=1}^2 G_{i+2}\gamma'_i\right]\nonumber\\
+2bc\left[f_3\beta_3\gamma_3G_4 + f_4\beta_4\gamma_4G_3\right]
+(-a-c^2)\left[f_3\beta^2_3G_4 + f_4\gamma^2_4G_3\right]\nonumber\\
+(-a-b^2)\left[f_3\gamma^2_3G_4 + f_4\gamma^2_4G_3\right]=0.
\end{eqnarray}
\end{corollary}

\noindent
To end this section, we will give the operator of Laplace-Beltrami in the following theorem:

\begin{theorem}
The components of the Laplace-Beltrami operator of the $2$-ruled hypersurface of type-$2$ are
\begin{eqnarray}\label{eq4.9}
\Delta\varphi_i=\frac{1}{\sqrt{T} }\left[
\begin{aligned}
&\frac{(\alpha''_i + u_2\beta''_i + u_3\gamma''_i) + (b\beta_i)_{u_1}
+ (c\gamma_i)_{u_1})T
- R_1(\alpha'_i + u_2\beta'_i + u_3\gamma'_i + b\beta_i + c\gamma_i )}{T^{\frac{3}{2}}}\\
& +\frac{(b\beta'_i + ((-a-c^2)\beta_i)_{u_2} + (bc\gamma_i)_{u_2})T
- R_2(b(\alpha'_i + u_2\beta'_i + u_3\gamma'_i) + (-a-c^2)\beta_i
+ bc\gamma_i)}{R^\frac{3}{2}} \\
&+ \frac{(c\gamma'_i + (bc\beta_i)_{u_3} + ((-a-b^2)\gamma_i)_{u_3})T
- R_3(c(\alpha'_i + u_2\beta'_i + u_3\gamma'_i) + bc\beta_i 
+ (-a-b^2)\gamma_i)}{R^\frac{3}{2}}
\end{aligned}
\right],
\end{eqnarray}
where $i = 1, 2, 3, 4$; $\beta$ and $\gamma$ are orthogonal; $T=a+b^2+c^2$, $R_1=a_{u_1}+2bb_{u_2}+2cc_{u_1}$, $R_2=a_{u_2}+2bb_{u_2}+2cc_{u_2}$, $R_3=a_{u_3}+2bb_{u_3}+2cc_{u_3}$.
\end{theorem}

\subsection{2-Ruled hypersurfaces of type-3 in $M$}

\noindent
A $2$-ruled hypersurface of type-$3$ in $M$  means (the image of) a map
$\varphi:I_1\times I_2\times I_3\to M$ of the form
\begin{eqnarray}\label{eq5.1}
\varphi(u_1,u_2,u_3)=\alpha(u_1)+u_2\beta(u_1)+u_3\gamma(u_1),
\end{eqnarray}
where $\alpha: I_1\to M$,  $\beta: I_2\to \mathcal{LC}$, 
$\gamma : I_3 \to \mathcal{LC}$ are smooth maps, $\mathcal{LC}$ is the light cone 
of $M$ and $I_1, I_2, I_3$ are open intervals. We call $\alpha$ a base curve, 
$\beta$ and $\gamma$ director curves. The planes 
$(u_2,u_3)\mapsto \alpha(u_1) + u_2\beta(u_1) + u_3\gamma(u_1)$ are 
called rulings. So, if we take
\begin{eqnarray}\label{eq5.2}
\left \{  \begin{array}{llllll}
\alpha(u_1) &=& \big(\alpha_1(u_1), \alpha_2(u_1), \alpha_3(u_1),
\alpha_4(u_1)\big),  \\
\beta(u_1) &=& \big(\beta_1 (u_1), \beta_2 (u_1), \beta_3(u_1),
\beta_4(u_1)\big),   \\
\gamma(u_1) &=& \big(\gamma_1(u_1), \gamma_2(u_1), \gamma_3(u_1),\gamma_4(u_1)\big),
\end{array}
\right.
\end{eqnarray}
in (\ref{eq5.1}), then we can write the $2$-ruled hypersurface of type-$3$ as
\begin{eqnarray}\label{eq5.3}
\varphi(u_1,u_2,u_3) =\left(
\begin{array}{cc}
\alpha_1(u_1)+u_2\beta_1(u_1)+u_3\gamma_1(u_1), & \alpha_2(u_1)+u_2\beta_2(u_1)+u_3\gamma_2(u_1),\\
\alpha_3(u_1)+u_2\beta_3(u_1)+u_2\gamma_3(u_1), & \alpha_4(u_1)+u_2\beta_4(u_1)+u_3\gamma_4(u_1)
\end{array}
  \right).
\end{eqnarray}

\noindent
We see that $\langle\beta_i,\beta_i\rangle=\langle\gamma_i,\gamma_i\rangle=-1$ and we state $\alpha_i=\alpha_i(u_1)$, 
$\beta_i=\beta_i(u_1)$, $\gamma_i=\gamma_i(u_1)$, $\varphi_i=\varphi_i(u_1,u_2,u_3)$, $f'=\frac{\partial f(u_1)}{\partial u_1}$, $f''=\frac{\partial^2f(u_1)}{\partial u_1\partial u_1}$, $i\in \{1,2,3,4\}$ and 
$f\in \{\alpha,\beta,\gamma\}$.\\

\noindent
From (\ref{eq2.8}), we obtain the matrix of the first fundamental form
\begin{eqnarray}\label{eq5.4}
[g_{ij}]=
\left[
   \begin{array}{ccc}
    a & b & c \\
   b & 0 & e \\
   c & e & 0
   \end{array}
\right].
\end{eqnarray}
And we obtain the inverse matrix $[g^{ij}]$ of $[g_{ij}]$ as
\begin{eqnarray}\label{eq5.5}
[g^{ij}]=\frac{1}{\det[g_{ij}]}
\left[
   \begin{array}{ccc}
    -e^2 & ce & be \\
  ce & -c^2 & bc-ae \\
  be & bc-ae & -b^2
   \end{array}
\right],
\end{eqnarray}
where
$a$, $b$, $c$ and $e$ are the same in (\ref{eq3.10}) and
\begin{eqnarray}\label{eq5.6)}
\det[g_{ij}]=2bce - ae^2 = D.
\end{eqnarray}
Furthermore, from (\ref{eq2.9}), the matrix form of the second fundamental from of the $2$-ruled hypersurface (\ref{eq5.3}) is the same given in 
(\ref{eq3.13}) and (\ref{eq3.14}). And we have the following theorem since 
the $\det [h_{ij}]\neq 0$.

\begin{theorem}
The $2$-ruled hypersurfaces of type-$3$ defined in (\ref{eq5.3}) is no flat.
\end{theorem}

\begin{corollary}
The $2$-ruled hypersurfaces of type-$3$ is flat if $f$ is non zero constant.
\end{corollary}

\noindent
For the mean curvature we have

\begin{theorem}
The $2$-ruled hypersurfaces of type-$3$ defined in (\ref{eq5.3}) is minimal 
in $M$, if
\begin{eqnarray}\label{eq5.7}
-e^2\left[f_3G_4(\alpha_3' + u_2\beta_3' + u_3\gamma_3')+f_4G_3(\alpha_4' 
+ u_2\gamma_4' + u_3\gamma_4') + \sum^{2}_{i=1}G_{i+2}(\alpha_i'' 
+ u_2\beta_i'' + u_3\gamma_i'')\right]\nonumber\\
+2ce\left[f_3\beta_3G_4(\alpha'_3 + u_2\beta'_3 + u_3\gamma'_3)
+ f_4\beta_4G_3(\alpha'_4 + u_2\beta'_4 + u_3\gamma'_4)
+ \sum_{i=1}^2 G_{i+2}\beta'_i + \right]\nonumber\\
+ 2be\left[f_3\gamma_3G_4(\alpha'_3 + u_2\beta'_3 + u_3\gamma'_3)
+ f_4\gamma_4G_3(\alpha'_4 + u_2\beta'_4 + u_3\gamma'_4)
+ \sum_{i=1}^2 G_{i+2}\gamma'_i\right]\nonumber\\
+ 2(bc-ae)\left[f_3\beta_3\gamma_3G_4 + f_4\beta_4\gamma_4G_3\right]
- c^2\left[f_3\beta^2_3G_4 + f_4\beta^2_4G_3\right]\nonumber\\
- b^2\left[h_3\gamma^2_3G_4 + h_4\gamma^2_4G_3\right]=0.
\end{eqnarray}
\end{theorem}

\begin{proof}
By (\ref{eq2.10})  the matrix of the shape operator is
\begin{eqnarray*}
S=
\left[
   \begin{array}{ccc}
    -e^2 & ce & be \\
  ce & -c^2 &bc-ae \\
  be & bc-ae & -b^2
   \end{array}
\right]\left[
   \begin{array}{ccc}
    h_{11} & h_{12} & h_{13} \\
   h_{12} & h_{22} & h_{23} \\
 h_{13} & h_{23} & h_{33}
   \end{array}
\right],
\end{eqnarray*}
where $h_{11}$, $h_{12}$, $h_{13}$, $h_{22}$, $h_{23}$, $h_{33}$ are the 
same in (\ref{eq3.14}). Then we get the coefficients of $S$ by
\begin{eqnarray*}
S_{11}&=&-e^2h_{11}+ceh_{12}+beh_{13},\\
S_{22}&=&ceh_{12}-c^2h_{22}+(bc-ae)h_{23},\\
S_{33}&=&be+h_{13}+(bc-ae)h_{23}-b^2h_{33}.
\end{eqnarray*}
And using (\ref{eq3.14}) and (\ref{eq2.12}), we see that the $2$-ruled hypersurfaces of type-$3$ is minimal if
\begin{eqnarray*}
S_{11}+S_{22}+S_{33}=0,
\end{eqnarray*}
then that end the proof.
\end{proof}

\begin{corollary}
If the curves $\beta$ and $\gamma$ are orthogonal then the $3$-ruled hypersurfaces of type-$3$ defined in (\ref{eq5.3}) is minimal if
\begin{eqnarray}\label{eq5.8}
+2bc\left[f_3\beta_3\gamma_3G_4+f_4\beta_4\gamma_4G_3\right]
-c^2\left[f_3\beta^2_3G_4+h_4\gamma^2_4G_3\right]-b^2\left[f_3\gamma^2_3G_4+f_4\gamma^2_4G_3\right]=0.
\end{eqnarray}
\end{corollary}

\noindent
To end this section, we will give the operator of Laplace-Beltrami in the following theorem:

\begin{theorem}
The components of the Laplace-Beltrami operator of the $2$-ruled hypersurface of type-$3$ are
\begin{eqnarray}\label{eq5.9}
\Delta\varphi_i=\frac{1}{\sqrt{L} }\left[
\begin{aligned}
&\frac{(\alpha''_i+u_2\beta''_i+u_3\gamma''_i)+(b\beta_i)_{u_1}+(c\gamma_i)_{u_1})L-J_1(\alpha'_i+u_2\beta'_i+u_3\gamma'_i+b\beta_i+c\gamma_i )}{L^{\frac{3}{2}}}\\
& +\frac{(b\beta'_i+((-a-c^2)\beta_i)_{u_2}+(bc\gamma_i)_{u_2})L-J_2(b(\alpha'_i+u_2\beta'_i+u_3\gamma'_i)+(-a-c^2)\beta_i+bc\gamma_i)}{L^\frac{3}{2}} \\
&+ \frac{(c\gamma'_i+(bc\beta_i)_{u_3}+((-a-b^2)\gamma_i)_{u_3})L-J_3(c(\alpha'_i+y\beta'_i+z\gamma'_i)+bc\beta_i+(-a-b^2)\gamma_i)}{L^\frac{3}{2}}
\end{aligned}
\right],
\end{eqnarray}
where $i = 1, 2, 3, 4$; ; $L=2bce-ae^22$, $J_1=2b_{u_1}ce+2bc_{u_1}e+2bce_{u_1}-2aa_{u_1}$, $J_2=2b_{u_2}ce+2bc_{u_2}e+2bce_{u_2}-2aa_{u_2}$, $J_3=2b_{u_3}ce+2bc_{u_3}e+2bce_{u_3}-2aa_{u_3}$.
\end{theorem}

\noindent
Note that the hypersurfaces constructed in this paper are not flat. Unlike Euclidean and Monkowskian spaces, where the ruled hypersurfaces are flat

\section{Conclusion}
\noindent
We end this work by giving some applications of ruled surfaces and $2$-ruled
hypersurfaces as eneralisations of the first one. Ruled surfaces have been applied in different areas such as CAD, electric discharge machining \cite{Altin2021, Ravani1991-2}. The authors \cite{Ravani1991-1} present an elementary introduction to the theory of Betrand pairs of curves and ruled surfaces. Bertrand pairs of ruled surfaces are introduced as offsets in the context of line geometry. Also, the ruled surfaces has an important application area on kinematics \cite{Altin2021,Pottmann2001}. Additionally, ruled surfaces have an important application area in arcitecture \cite{Altin2021}. For lightweight structures in the field of architecture and civil engineering, concrete shells with negative Gaussian curvature are frequently used. One class of such surfaces are the skew ruled surfaces \cite{Noak2019}.

\section*{Acknowledgments}
The authors would like to thank the referee for his/her valuable suggestions and comments that helped them improve the paper.

\end{document}